\documentclass[10pt]{article}
\usepackage[OT2,T1]{fontenc}
\usepackage{amssymb,amsfonts,amsmath,amsthm,amscd}
\usepackage{relsize}
\usepackage{authblk}
\usepackage[margin=2.8cm]{geometry}
\usepackage{bigstrut}
\usepackage{mathtools}
\usepackage{titlefoot}
\usepackage{cite}

\DeclareMathOperator{\ad}{ad}
\DeclareMathOperator{\Aut}{Aut}
\DeclareMathOperator{\tr}{tr}
\DeclareMathOperator{\diag}{diag}

\def\RR{\mathbb{R}}

\def\gg{\mathfrak{g}}

\def\ZZ{\mathcal{Z}}
\def\JJ{\mathcal{J}}
\def\t{\mathsmaller{T}}

\def\hn{\mathfrak{h}_{2n+1}}
\def\hnz{\mathfrak{h}_{2n+1}^*}
\def\thn{\rm{T}^*\mathfrak{h}_{2n+1}}
\def\th3{\rm{T}^*\mathfrak{h}_{3}}

\newtheorem{lemma}{{Lemma}}[section]
\newtheorem{theorem}{{Theorem}}[section]

\theoremstyle{definition}
\newtheorem{remark}{{Remark}}[section]

\newtheorem*{thank}{{\indent Acknowledgement}}

\title{Geometry of cotangent bundle of Heisenberg group}
\author{Tijana \v{S}ukilovi\'{c}\thanks{Corresponding author: tijana@matf.bg.ac.rs},  Sr\dj{}an Vukmirovi\'{c}}
\affil{\small{University of Belgrade, Faculty of Mathematics, Belgrade, Serbia}}
\date{\today}

\begin{document}
\maketitle\unmarkedfntext{2020 \emph{Mathematics Subject Classification}. Primary 22E25, 32Q15; Secondary 53C55\\ \indent\phantom{K}\emph{Key words and phrases}: cotangent bundle, Heisenberg group,  left invariant metrics, pseudo-K\"{a}hler metrics.}

\begin{abstract}
In this paper the classification of left-invariant Riemannian  metrics, up to the action of the automorphism group, on cotangent bundle  of (2n+1)-dimensional Heisenberg group is presented.
Also, it is proved that the complex structure on that group is unique and the corresponding pseudo-K\"{a}hler metrics are described and shown to be Ricci flat.
It is well known that this algebra admits an ad-invariant metric of neutral signature. Here, the uniqueness of such metric is proved.
\end{abstract}

\section*{Introduction}

Let us study the simply connected Lie group $G$. By the \emph{moduli space} we consider the orbit space of the action of $\RR^\times\Aut(\gg)$ on the space $\mathfrak{M}(G)$ of left invariant metrics on $G$, represented by its inner products on the corresponding Lie algebra $\gg$. Here, $\Aut(\gg)$ denotes the automorphism group of the corresponding Lie algebra and its orbits induce the isometry classes, while $\RR^\times$ is the scalar group giving rise to the scaling.

There are two, in some way dual, approaches to the classification problem. The first one is to start from the (pseudo-)orthonormal basis of Lie algebra and act by the isometry group (and hence preserving the canonical form of the metric) on the commutators to make them as simple as possible. The second one is to start from basis of Lie algebra with simple commutators and act by the automorphism group of the Lie algebra (hence preserving the commutators) to make the metric take the most simple form. For the more detailed outline of each approach, we refer to~\cite{Tamaru1, Tamaru2}.

The moduli space approach to the problem of classification was first introduced in Milnor's classical paper~\cite{Milnor} where left invariant Riemannian metrics on three-dimensional unimodular Lie groups were considered. Much later, the classification was completed in Lorentz case~\cite{CorderoParker}. In higher dimensions, there are numerous results in both Riemannian and pseudo-Riemannian setting. However, even in small dimensions, the moduli space of pseudo-Riemannian metrics can be quite large, hence is rarely practical or possible to obtain the complete classification of metrics.

Naturally, the case of nilpotent Lie groups is very thoroughly investigated (for example, see~\cite{Lauret1,Bokan1,Sukilovic1,Homolya, Eberlein,Vuk2015}). Following our previous interests, in this paper the cotangent bundle of the Heisenberg group $H_{2n+1}$ is considered. Since the Heisenberg group is two-step nilpotent, its cotangent bundle has the same property. Although the moduli space of metrics on the cotangent bundle can be constructed using both nondegenerate and degenerate metrics on the original Lie group, in practice the process heavily depends on the subtle geometrical and algebraic analysis. For example, in case of Lie group $T^*H_3$ we obtained 19 non-equivalent families of left invariant metrics of arbitrary signature (see~\cite{SVB} for more details). Therefore, in higher dimensions it is feasible to consider Riemannian signature and only some special cases of pseudo-Riemannian metrics.

The paper is organized as follows.

First, we briefly recall the construction of the cotangent bundle $\thn$ of the Lie algebra $\hn$ corresponding to the Heisenberg group $H_{2n+1}$. The explicit form of the automorphism group is described in Lemma~\ref{le:auto}.

In Section~\ref{sec:class} we classify all non isometric left invariant Riemannian metrics on $\thn$ (see Theorem~\ref{th:metrike}). Throughout the paper we identify the notion of metric on Lie group and the inner product on its Lie algebra. We obtained only one $n(2n+1)$-parameter family of Riemannian metrics. This family is correlated with the family of Riemannian metrics on Heisenberg group from~\cite{Vuk2015}.

Motivated by their applications in mathematics and physics, Section~\ref{sec:adinv} is devoted to ad-invariant metrics. Cotangent Lie algebras are endowed with the canonical ad-invariant metric. This metric is of neutral signature and since the setting is two-step nilpotent, it is flat. We prove the uniqueness of this metric in Lemma~\ref{le:adinv}.

 In Section~\ref{sec:pseudoK} we classify pseudo-K\"{a}hler metrics. Theorem~\ref{th:complex} states that, up to the action of automorphisms and sign, there exists only one complex structure $\JJ_0$ on $\thn$. This generalizes the result obtained by Salamon~\cite{Salamon} in his classification of complex structures on nilpotent Lie algebras, where he considered the Lie algebra $\th3$.
This complex structure is three-step nilpotent (see~\cite{Cordero}), but since the center of the algebra $\thn$ is odd dimensional, $\JJ_0$ cannot be abelian complex structure, i.e. complex structure satisfying $[x,y]=[\JJ x, \JJ y]$. In Theorem~\ref{th:forme} we find that the space of symplectic forms compatible with $\JJ_0$ has a dimension $\frac{3n^2 + n + 2}{2}$, $n>1$. Interestingly, for $n=1$, the dimension of $\JJ$-invariant closed 2-forms is five (see Remark~\ref{re:forme}). This case was considered in the paper~\cite{CSO}. Finally, we classify  pseudo-K\"{a}hler metrics and show that they all belong to the same family of Ricci-flat metrics (Theorem~\ref{th:pseudoK}).

\section{Preliminaries}
\label{sec:prelim}

Let us briefly recall the construction of cotangent  Lie algebra.

The \emph{cotangent algebra} $T^*\gg$ of the Lie algebra $\gg$ is a semidirect product of  $\gg$ and its cotangent space $\gg ^*$ by means of the coadjoint representation
\begin{align*}
	T^* \gg := \gg \ltimes_{\ad^*} \gg^*,
\end{align*}
i.e. the commutators are defined by
\begin{align}
	\label{eq:cotangent}
	[(x,\phi), (y,\psi)]:= ([x,y], \ad^*(x)(\psi) - \ad^*(y)(\phi)), \quad x,y\in \gg,\enskip \phi, \psi \in \gg^*.
\end{align}
The coadjoint representation  $\ad^* : \gg \to  \mathrm{gl} (\gg^*)$ is given by
\begin{align*}
	(\ad^*(x) (\phi))(y):= - \phi (\ad(x)(y)) = - \phi([x,y]).
\end{align*}
Now, we specialise this construction to the Heisenberg algebra $\hn$ which is the Lie algebra of (2n+1)-dimensional Heisenberg group $H_{2n+1}.$

Let us denote the matrix of standard complex structure on $\RR ^{2n}$ by
\begin{align}
	\label{eq:J}
J = J_{2n} =
\begin{pmatrix}
	0 & -E\\
	E & 0
\end{pmatrix}
\end{align}
where $E$ is $n\times n$ identity matrix. The standard symplectic form in vector space $\RR ^{2n}$ can be written in the form
\begin{align*}
\omega (x,y)  = x^\t Jy, \enskip x,y \in \RR ^{2n}.
\end{align*}
The Heisenberg group is two-step nilpotent Lie group $H_{2n+1}$  defined on the base manifold $\RR ^{2n}\oplus \RR$ by multiplication
\begin{align*}
(x, \mu )\cdot (y, \lambda):= (x + y , \mu + \lambda + \omega (x,y)).
\end{align*}
The corresponding Lie algebra
\begin{align*}
\hn = \RR ^{2n}\oplus \RR = \RR ^{2n}\oplus \ZZ = \{ (x, \lambda)  \, | \,  x\in R^{2n}, \lambda \in \RR\}
\end{align*}
is given by the following commutator equation
\begin{align}
	\label{eq:hnkom}
	[(x, \mu ),(y, \lambda)] = (0, \omega (x,y)).
\end{align}
Note that $\ZZ = \RR \langle z\rangle $ is one-dimensional center and one-dimensional commutator subalgebra of $\hn .$ To simplify notation we
write~\eqref{eq:hnkom} in the equivalent form
\begin{align*}
	[x + \mu z, y + \lambda z] =  \omega (x,y) z.
\end{align*}
If  $e_1, \dots , e_n, f_1, \dots , f_n$ represents the standard basis of $\RR ^{2n}$, then nonzero commutators of $\hn$ are
\begin{align*}
[e_i, f_i] = z, \enskip i = 1, \dots , n.
\end{align*}
Let
$ \RR ^{*2n} = \RR \langle e_1^*, \dots , e_n^*, f_1^*, \dots , f_n^* \rangle$  and $\ZZ ^* = \RR \langle z^* \rangle$
be dual vector spaces of one-forms spanned by dual basis vectors.
Denote the dual space of $\hn$ by
\begin{align*}
\hnz =   \RR ^{*2n}\oplus \ZZ^* = \{ (x^*, \mu ^*)  \, | \,  x^*\in R^{2n}, \mu ^* \in \RR\} = \{ x^* + \mu ^* z^*\}.
\end{align*}
Now, the cotangent space $\thn$ of $\hn$, as a vector space, can be written as a direct sum
\begin{align}
	\label{eq:thn_decomp}
\gg = \thn = \hn \oplus \hnz = \RR^{2n} \oplus \ZZ^* \oplus \RR ^{*2n}\oplus \ZZ \cong \RR^{4n+2}.
\end{align}
Note, that we changed the order of summands to better fit the structure of Lie algebra.
Namely, one can check that according to~\eqref{eq:cotangent} the commutator on $\thn$ is given by
\begin{align}
	\label{eq:thnkom1}
	[(x,\mu ^*, x^*, \mu), (y,\lambda  ^*, y^*, \lambda)] = (0,0, \lambda^*J^*(x)-\mu^* J^*(y), \omega(x,y)),
\end{align}
where
\begin{align*}
J^*: \RR ^{*2n}\to \RR ^{*2n}, \quad  J^*(X) := (JX)^*
\end{align*}
 is  represented by matrix  $J$ given in~\eqref{eq:J},  in basis $ e_1^*, \dots , e_n^*, f_1^*, \dots , f_n^*$.

The center and commutator subalgebra of $\thn$ coincide with $\RR ^{*2n}\oplus \ZZ $. By that reason we have changed the order of summands in~\eqref{eq:thn_decomp} to group central (and commutator) vectors together. Note that the algebra $\thn$ is also two-step nilpotent.

To simplify the notation we write commutator~\eqref{eq:thnkom1} of $\thn$ in equivalent form
\begin{align}
		\label{eq:thnkom2}
[x+ \mu ^* z^* +  x^*  +  \mu z, y + \lambda  ^* z^* +  y^* +  \lambda z] = \lambda^*J^*(x)-\mu^* J^*(y) +  \omega(x,y) z.
\end{align}

\begin{lemma}
	\label{le:auto}
	The group of automorphism of algebra $\gg = \thn$ of the form~\eqref{eq:thn_decomp} in standard basis is
	\begin{align}
		\label{eq:autoF}
	&\Aut (\thn ) = \left\{
			F = \begin{pmatrix}
			F_1 & 0  \\
			F_3 & F_4
		\end{pmatrix} \, \vert\, \, F_1, F_4\in Gl_{2n+1}(\RR ), \, F_3 \in M_{2n+1}(\RR )\,  \right\},
\\
		\label{eq:autoF1}
&F_1=	\begin{pmatrix}
		\bar F_1 & v_1  \\
		u_1^\t  & f_1
	\end{pmatrix},\quad
F_4=	\begin{pmatrix}
	f_1 f_4 {\bar F_1}^{-\t}  - (Jv_1)(Ju_1)^\t  & -f_4 {\bar F_1}^{-\t}  u_1   \\
	-f_4v_1^\t {\bar F_1}^{-\t}  & f_4
\end{pmatrix},\quad \bar F_1^\t J\bar F_1 = f_4 J,
\end{align}
 $\enskip v_1, u_1 \in \RR ^{2n}, \, f_1, f_4 \in \RR\backslash\{0\}. $
	 Its dimension is
	 $\dim \Aut (\thn ) =  6 n^2 + 9n + 3.$
\end{lemma}
\begin{proof}
	For element $g \in  \gg$ denote by $A_g$ the matrix of  operator $ad(g)$ in standard basis.
Each $g\in \gg$ decomposes as
$g = \nu + \zeta$, $\nu \in \RR ^{2n}\oplus \ZZ^*, \zeta \in \RR ^{*2n}\oplus \ZZ.$ Note that  $ad (\zeta ) \equiv 0$ for the central part. If $\nu =  x + \mu^* z^*$ the matrix of $ad (g)$ is
  \begin{align}
  	\label{eq:uslov}
		 A_g = 	\begin{pmatrix}
			0 & 0  \\
			A_\nu & 0
		\end{pmatrix},\quad A_\nu =
\begin{pmatrix}
	-\mu^* J & Jx  \\
	x^\t J & 0
\end{pmatrix}.
	\end{align}
	The condition that $F\in \Aut (\gg)$ can we written in the eqivalent form
	\begin{align}
		\label{eq:auto}
		F[g,h] = [Fg, Fh], \enskip g,h \in \gg \quad \Longleftrightarrow \quad A_{Fg}F = FA_g, \enskip g\in \gg.
	\end{align}
We are looking for the matrix of automorphism in a block matrix  form
\begin{align*}
F = \begin{pmatrix}
	F_1 & F_2  \\
	F_3 & F_4
\end{pmatrix}.
\end{align*}
Automorphisms preserve the center  $ \RR ^{*2n}\oplus \ZZ$ of $\gg$ and therefore $F_2=0.$ Relations~\eqref{eq:auto} impose no condition on $F_3$ and hence $F_3\in M_{2n+1}(\RR)$ is arbitrary. The remaining condition is
\begin{align}
	\label{eq:uslov1}
	A_{F_1\nu}F_1 = F_4 A_\nu \enskip \mbox{for each}\enskip \nu =  x + \mu^* z^*.
\end{align}
We further decompose matrices $F_1$ and $F_4$ as
\begin{align}
	\label{eq:F14}
	F_k = \begin{pmatrix}
		\bar F_k & v_k  \\
		u_k^\t  & f_k
	\end{pmatrix},
\end{align}
 $\bar F_k\in M_{2n}(\RR), v_k, u_k \in \RR^{2n}, f_k \in \RR, k=1,4.$
Next, we use condition~\eqref{eq:uslov1} for matrices~\eqref{eq:F14} and operator~\eqref{eq:uslov}. The bottom left component is
\begin{align*}
	(\bar F_1 X  + \mu^*v_1)^\t  J\bar F_1 = (-\mu^*u_4^\t  + f_4 X^\t ) J
\end{align*}
and has to be satisfied for $x\in \RR^{2n}$ and $\mu^* \in \RR.$ It is equivalent to the following two conditions
\begin{align}
	\label{eq:uslovidolelevo}
 	\bar F_1^\t J\bar F_1 = f_4 J, \quad\quad  u_4 = -f_4\bar F_1^{-1} v_1.
\end{align}
The bottom right component is
\begin{align*}
	(\bar F_1X)^\t Jv_1 = u_4^\t JX
\end{align*}
and because of relations~\eqref{eq:uslovidolelevo} it is satisfied for any $v_1, X\in \RR^{2n}.$ Therefore, $v_1\in \RR^{2n}$ is arbitrary.

The upper left and upper right components are, respectively
\begin{align*}
	-(u_1^\t X + \mu^*f_1)J\bar F_1 + J(\bar F_1X + \mu^*v_1)\mu_1^\t  = (-\mu^*\bar F_4 + v_4x^\t ) J, \quad \bar F_4JX = J(-u_1^\t Xv_1 + f_1\bar F_1X),
\end{align*}
for all $X\in \RR^{2n}.$ Combining them with previous relations yields
\begin{align}
	\label{eq:uslovigore}
	\bar F_4 = f_1 f_4 {\bar F_1}^{-\t}  - (Jv_1)(Ju_1)^\t ,\quad \quad v_1 = -f_4 {\bar F_1}^{-\t}  u_1
\end{align}
and $u_1\in \RR ^{2n}$ is arbitrary. Numbers $f_1, f_4$ can't be equal to zero since $F$ is nonsigular. The relations~\eqref{eq:uslovidolelevo} and~\eqref{eq:uslovigore} imply that the automorphisms has form as in the statement of the Lemma.

Two calculate the dimension of the group note that matrix $\bar F_1$ is ``almost'' symplectic and therefore depends on $n(2n+1)$ paremeters. Since $F_3\in M_{2n+1}(\RR)$, $u_1, v_1\in \RR ^{2n}$, $f_1,f_4\in \RR\backslash\{0\}$ are arbitrary, and $F_4$ is completely dependent on $F_1$, the dimension is:
\begin{align*}
n(2n+1) + (2n+1)^2 + 2n + 2n + 2 =  6 n^2 + 9n + 3.
\end{align*}
\end{proof}

\begin{remark}
	In~\cite{SVB} group of  automorphism of Lie algebra $\rm{T}^*\mathfrak{h}_{3}$ (special case for $n=1$) are given. Group of automorphisms of Heisenberg algebra $\hn$ is subgroup of $\Aut (\thn )$  and can be described as  semidirect product of symplectic group $Sp(2n, \RR),$ subgroup of translations isomorphic to $\RR ^{2n}$ and $1$-dimensional ideal. For more details, see~\cite{Vuk2015}.
\end{remark}

\section{Classification of left invariant Riemannian metrics}\label{sec:class}

In this section we  classify non-isometric left invariant Riemannian metrics on $\gg = \thn$.

If  $\gg$ is a Lie algebra and $\langle \cdot, \cdot \rangle$ inner product on $\gg$ the pair $(\gg, \langle \cdot, \cdot \rangle)$ is called a \emph{metric Lie algebra}.
The structure of metric Lie algebra uniquely defines left invariant Riemannian metric on the  corresponding simple connected Lie group $G$ and vice versa.

Two metric Lie algebras are \emph{isomorphic} if there are isomorphic as Lie algebras and that isomorphism is isometry of their inner products.

Metric algebras are said to be \emph{isometric} if there exists an isomorphism of Euclidean spaces preserving the curvature tensor and its covariant derivatives. This translates to the condition that metric algebras are isometric if and only if they are isometric as Riemannian spaces
(see~\cite[Proposition 2.2]{Aleksijevski}). Although two isomorphic metric algebras are also isometric, the converse is not true.
In general, two metric algebras may be isometric even if the corresponding Lie algebras are non-isomorphic. The test to determine whether any two given solvable metric algebras (i.e. solvmanifolds) are isometric was developed by Gordon and Wilson in~\cite{Gordon}. However, by the results of Alekseevski\u{\i}~\cite[Proposition 2.3]{Aleksijevski}, in the completely solvable case, isometric means isomorphic.

Since Lie algebra $\gg$ is nilpotent and therefore completely solvable, non-isometric metrics on $\gg$ are the non-isomorphic ones. They are obtained by action of group $\Aut (\gg)$  on space of metrics.

Fix basis of $\thn$ described in Section \ref{sec:prelim}. Riemannian  inner product on $\langle \cdot, \cdot \rangle$ is then represented by nonsingular symmetric matrix $S \in M_{4n+2}(\RR)$ of signature $(4n+2,0).$

The action of $F\in \Aut (\thn )$ on inner products is given by
\begin{align}
\label{eq:dejstvo}
	S' = F^\t SF.
\end{align}
Matrix $S$ has $(4n+3)(2n+1) = 8n^2+10n+3$ parameters. According to $\dim \Aut (\thn)$ from Lemma~\ref{le:auto} it is expected that moduli space of non-isometric inner products is of dimension $n(2n+1).$ This is what we show in Theorem~\ref{th:metrike}.

Write $S$ in block-matrix form
\begin{align*}
	S = \begin{pmatrix}
		S_1 & S_2  \\
		S_2^\t  & S_4
	\end{pmatrix}, \quad S_1^\t  = S_1, S_4^\t  = S_4.
\end{align*}
Let $F \in \Aut(\thn)$ be of the form~\eqref{eq:autoF}. After the action~\eqref{eq:dejstvo} the  top right matrix of $S'$  is
\begin{align*}
S_2' = F_1^\t (S_2F_4) + F_3^\t (S_4F_4).
\end{align*}
Matrix $S_4$ represents the restriction of the inner product on subspace $\RR ^{*2n}\oplus \ZZ.$ Since inner product is Riemannian, the restriction is non-degenerate and thus $S_4$ is nonsingular.
By choosing $F_3 = -S_2F_1S_4^{-1}$ we achieve $S_2'=0$ and therefore we can suppose that
\begin{align}
	\label{eq:s1s4}
S = S' = \begin{pmatrix}
	S_1 & 0  \\
	0 & S_4
\end{pmatrix}, \quad S_1^\t  = S_1,\quad S_4^\t  = S_4.
\end{align}
Now, the action of $F$ on $S$, with $F_3=0$ is given by
\begin{align*}
	S_k' = F_k^\t  S_k F_k, \enskip  k= 1,4,
\end{align*}
where $F_1$ and $F_4$ are related by~\eqref{eq:autoF1}.
We write  $(2n+1)\times (2n+1)$  matrices $F_k$ and  $S_k$ as block matrices
\begin{align*}
	F_k = \begin{pmatrix}
	\bar F_k & v_k  \\
		u_k^\t  & f_k
	\end{pmatrix}, \quad \quad
S_k = \begin{pmatrix}
	\bar S_k & t_k  \\
	t_k^\t  & \omega _k
\end{pmatrix}, \quad {\bar S_k}^\t  = \bar S_k, t_k\in \RR ^{2n}, \quad k=1,4.
\end{align*}

The top right vector in $S_k'$ given is given by
\begin{align*}
t_k' = u_k( t_k^\t v_k + \omega _k f_k) + F_k^\t \bar S_k v_k + \bar F_k^\t t_kf_k , \quad k=1,4.
\end{align*}
The number $\omega _k,\, k =  1,4$ is different from zero (it is the norm of the vector $z^*$ and $z$, respectively).  With appropriate  choice of number $f_k$ one can achieve that  $t_k^\t v_k + \omega _k f_k \neq 0,\ k=1,4.$

Hence, there is a choice automorphism $F$ with  $u_1$ such that top right vector $t_1'$ in $S_1'$ is zero and with $u_4,$ or eqivalently $v_1$, such that top right vector $t_4'$ in $S_4'$ is zero.

Therefore, we can suppose that metric has form~\eqref{eq:s1s4} with
\begin{align*}
	S_k = \begin{pmatrix}
	\bar S_k & 0   \\
	0 & \omega _k
\end{pmatrix}, \quad {\bar S_k}^\t  = \bar S_k, \quad k=1,4.
\end{align*}

To further simplify the metric consider automorphism $F$  of the form~\eqref{eq:autoF} with $u_1=v_1 =0$ (and hence $u_4=v_4 =0$). The action reduces to
\begin{align*}
		S_k' =
		\begin{pmatrix}
		{\bar F_k}^\t \bar S_k \bar F_k & 0   \\
		0 & f_k^2 \omega _k
	\end{pmatrix},\enskip k=1,4,
\end{align*}
with $F_1^\t J\bar F_1 = f_4 J.$

As first, consider upper left block of $S$, i.e. $k=1.$ Symplectic matrix $\bar F_1$ can diagonalize positive symmetric matrix $\bar S_1$ (see~\cite{Williamson,Vuk2015} for details regarding independent action of $\bar S_1$ and $f_4$) to obtain
\begin{align*}
	{\bar F_k}^\t \bar S_k \bar F_k = D(\sigma) = \diag (\sigma _1, \ldots , \sigma_ n, \sigma _1, \ldots , \sigma _n), \enskip \sigma _1\geq \ldots\geq \sigma _n >0.
\end{align*}
By choosing $f_4 = \frac{1}{\sigma _n}$ one can achieve $\sigma _n=1$ and for  $f_1 = \frac{1}{\sqrt{\omega _1}}$ we obtain
\begin{align}
	\label{eq:s1p}
	S_1' = \diag (D(\sigma),1) = \diag (\sigma _1, \dots , \sigma _{n-1}, 1, \sigma _1, \dots \sigma _{n-1}, 1,1).
\end{align}
Now consider bottom right block, i.e. $k=4.$ Since $\bar F_1, f_1,f_4$ and hence $\bar F_4,$ are already chosen,  there is no much freedom left. However, symplectic rotations $R_\phi:\RR ^{2n}\to \RR^{2n}$
\begin{align}\label{eq:srot}
	e_i  \mapsto \cos \phi_i\, e_i -\sin \phi _i\, f_i,\qquad
	f_i  \mapsto \sin \phi_i\, e_i +\cos \phi _i\, f_i,
\end{align}
by angles $\phi _1, \dots, \phi _n \in \RR$ belong to $\Aut(\hn) \subseteq \Aut (\thn)$ and preserve  inner product~\eqref{eq:s1p}. These rotations induce rotations in $\RR^{*2n}.$ By choosing appropriate angles one can achieve that $e_i^*$ and $f_i^*$ are orthogonal, i.e.
$(\bar S_4')_{i(n+i)} = 0$, $i = 1, \dots ,n$.
Free parameters in the canonical form of the metric are: $n-1$ parameter $\sigma _i,$ entries of symmetric $2n\times 2n$ matrix $\bar S_4'$ ($n$ of them is equal to zero) and parameter $\omega _4.$ Therefore the dimension of moduli space of non-equivalent metrics is
\begin{align*}
(n-1) + (n (2n+1) - n) + 1 = n(2n+1)
\end{align*}
as expected. Hence, we proved the following theorem.
\begin{theorem}
	\label{th:metrike}
Dimension of the moduli space of Riemannian metrics on Lie algebra $\thn$ is $n(2n+1).$ Every such metric is represented by $(4n+2)\times (4n+2)$ block matrix
\begin{align}\label{eq:metrike}
	S = \begin{pmatrix}
	D(\sigma) & 0 & 0 & 0   \\
	0 & 1 &  0 & 0\\
	0 & 0 & \bar S_4 & 0\\
	0 & 0 & 0 &   \omega _4
\end{pmatrix},
\end{align}
$\omega _4>0$, $ D(\sigma) = \diag (\sigma _1, \dots , \sigma_{n-1}, 1, \sigma _1, \dots , \sigma _{n-1}, 1), \enskip \sigma _1 \geq \dots \geq \sigma _{n-1} \geq 1$, $\bar S_4$ is symmetric positive definite matrix of dimension $2n\times 2n$ satisfying $(\bar S_4)_{i(n+i)} = 0, \enskip i=1,\dots n.$
\end{theorem}
\begin{remark}
  Symplectic rotations~\eqref{eq:srot} represent unique automorphisms preserving~\eqref{eq:s1p} if all $\sigma_k$ are distinct. If some of them are equal, there exist wider class of automorphism preserving~\eqref{eq:s1p} that further simplifies the matrix $\bar S_4$.
\end{remark}

\section{Ad-invariant metrics}\label{sec:adinv}

The next step would be to consider pseudo-Riemannian case. However, as previously mentioned, the corresponding moduli space of metrics will grow quite large with the increase of the dimension. Nevertheless, it is known that every cotangent bundle of a Lie group $G$ admits an ad-invariant metric defined by the duality pairing on the corresponding Lie algebra $T^*\gg=\gg\ltimes\gg^*$:
\begin{align}\label{eq:adinv}
\langle (x, x^*), (y,y^*)\rangle = x^*(y) + y^*(x),\enskip x,y\in\gg,\ x^*,y^*\in\gg^*.
\end{align}
The metric~\eqref{eq:adinv} is of neutral signature and it makes both algebras $\gg$ and $\gg^*$ totally isotropic, meaning $\gg^\perp = \gg$ and ${\gg^*}^\perp = \gg^*$.
\begin{lemma}\label{le:adinv}
  The metric~\eqref{eq:adinv} is unique ad-invariant metric on $\thn$ and it is flat.
\end{lemma}
\begin{proof}
  The condition of ad-invariance:
  \begin{align*}
    \langle[x,y],z]\rangle + \langle y, [x,z]\rangle = 0,\enskip x,y,z\in\thn,
  \end{align*}
  gives us that the corresponding metric is given by the symmetric matrix:
  \begin{align*}
    S=\begin{pmatrix}
      \bar S & \alpha E\\
      \alpha E & 0
    \end{pmatrix}, \enskip \bar S = \bar S^\t\in M_{2n+1}(\RR),\ \alpha\neq 0.
  \end{align*}
If we apply the automorphism $F$ of the form~\eqref{eq:autoF}-\eqref{eq:autoF1} with
\begin{align*}
	F_1 = E,\enskip F_4 = \frac{1}{\alpha} E, \enskip
	F_3 = -\frac{1}{2\alpha}\bar S,
\end{align*}
we obtain that $\bar S = 0$ and $\alpha = 1$, i.e. every ad-invariant metric is equivalent to the metric~\eqref{eq:adinv}.

The flatness of the metric follows trivially from the fact that the curvature tensor $R$ for ad-invariant metric is given by
  \begin{align*}
    R(u,v) = -\frac{1}{4}\ad_{[u,v]},\enskip u,v\in\thn.
  \end{align*}
  Therefore, the ad-invariant metric is flat if and only if the corresponding Lie algebra is two-step nilpotent.
\end{proof}

The previous results confirm the much extensive results on the uniqueness of the ad-invariant metrics recently obtained in~\cite{Conti}.

\section{Classsification of pseudo-K\"{a}hler metrics on $\thn$}\label{sec:pseudoK}

Almost complex structure $\JJ$ is linear map $\JJ :\thn \to \thn$ that satisfies $\JJ ^2 = -Id.$ 	
Recall that complex structure $\JJ$ is integrable complex structure, i.e. it satisfies $N_\JJ (X,Y) = 0$ for all $X,Y \in \thn $ where $N_\JJ$ is the Nijenhuis tensor defined  by
\begin{align}
	\label{eq:nijenuis}
	N_{\JJ}(X,Y) = [X,Y] + \JJ [\JJ X, Y] + \JJ [X,\JJ Y] - [\JJ X, \JJ Y].
\end{align}

If $\Omega$ is 2-form and $\JJ$ is an almost complex structure on $\gg$ we say that $\Omega$ is $\JJ$-invariant if
\begin{align}
	\label{eq:Jinvariant}
	\Omega (X,Y) = \Omega (\JJ X, \JJ Y),
\end{align}
for all $X,Y, \in \gg.$

The metric $\langle \cdot, \cdot \rangle $ defined by complex structure $\JJ$ and $\JJ$-invariant closed 2-form $\Omega$ by
\begin{align}\label{eq:herm}
	\langle X, Y \rangle := \Omega (\JJ  X, Y)
	\end{align}
is called pseudo-K\"ahler.

\begin{theorem}
	\label{th:complex}
Up to automorphisms the complex structure on $\thn$ is up to a sign unique, and represented by matrix
\begin{align}
	\label{eq:complex}
	\JJ _0=
	\begin{pmatrix}
		J & 0 & 0 & 0   \\
		0 & 0 &  0 & 1\\
		0 & 0 & J & 0\\
		0 & -1 & 0 &  0
	\end{pmatrix},
\end{align}
where $J$ is the $2n\times 2n$ matrix of standard complex structure given by~\eqref{eq:J}.
	\end{theorem}
\begin{proof}
Vanishing of Nienhuis tensor~\eqref{eq:nijenuis} can be equivalently expressed in terms of $\ad$ operator
\begin{align}
	\label{eq:complexAd}
	\ad(X) - \JJ \ad (\JJ X) - \JJ \ad (X)\JJ - \ad(\JJ X)\JJ = 0,
\end{align}
for all $X \in \thn.$

Each $X\in \thn$ decomposes as
$X = \nu + \zeta$, $\nu \in \RR ^{2n}\oplus \ZZ^*, \zeta \in \RR ^{*2n}\oplus \ZZ$
and the matrix $A_X\circ \ad (X)$ depends only on $\nu,$ as in~\eqref{eq:uslov}.
With respect to the same decomposition write the matrix of $\JJ$ in the form
\begin{align}
	\label{eq:Jblokzapis}
	\JJ = 	\begin{pmatrix}
		J_1 & J_2 \\
		J_3 & J_4 	
	\end{pmatrix}.
\end{align}
The condition~\eqref{eq:complexAd} is now equivalent to the following matrix equations:
\begin{align*}
	& J_2 A_\nu J_2 = 0, & &J_2A_{J_1\nu + J_2 \zeta} + J_2 A_\nu J_1 = 0,&\\
	& A_{J_1\nu + J_2\zeta}J_2- J_4A_\nu J_2 = 0, & &A_\nu + J_4A_{J_1\nu + J_2 \zeta} + J_4 A_\nu J_1 - A_{J_1\nu + J_2\zeta}J_1 = 0,&
\end{align*}
for all  $\nu \in \RR ^{2n}\oplus \ZZ^*, \zeta \in \RR ^{*2n}\oplus \ZZ.$
Careful analysis of these equations yields the following solution
\begin{align}
	\label{eq:usloviN}
		J_1 = \begin{pmatrix}
		\varepsilon J & 0 \\
		0 & n_1 	
	\end{pmatrix}, \quad  J_2 = \begin{pmatrix}
	 0 & 0 \\
	0 & n_2 	
\end{pmatrix},
 \quad  J_3 = \begin{pmatrix}
	\bar  J_3  & k_3 \\
	m_3^\t & n_3 	
\end{pmatrix},
\quad  J_4 = \begin{pmatrix}
	\varepsilon  J  & 0 \\
	0  & 0 	
\end{pmatrix},
\end{align}
with $\varepsilon = \pm 1, \, n_1, n_2, n_3 \in \RR, \, k_3, m_3 \in \RR^{2n}, \, \bar J_3 \in M_{2n}(\RR).$
If we apply condition $\JJ ^2 = - Id$, we obtain
\begin{align}
	\label{eq:usloviJ2}
	m_3 = k_3 = 0, \quad n_1 = 0, \quad n_3 = -\frac{1}{n_2}, \quad \bar J_3 J + J \bar J_3 = 0.
\end{align}
Now, we apply the automorphism $F$ of the form~\eqref{eq:autoF}-\eqref{eq:autoF1} with
\begin{align*}
	\bar F_1 = J,\enskip u_1 = v_1  = 0,\enskip f_1 = n_2, f_4 = \varepsilon, \enskip
	F_3 = \begin{pmatrix}
		  \frac{\bar J_3}{2}  & 0 \\
		0 & 0 	
	\end{pmatrix},
\end{align*}
to the structure $\JJ$ of the form~\eqref{eq:Jblokzapis} and~\eqref{eq:usloviN} satisfying~\eqref{eq:usloviJ2}, and obtain
\begin{align*}
	F^{-1} \JJ F = \varepsilon \JJ _0, \quad \varepsilon = \pm 1,
\end{align*}
as in the statement of the theorem.
\end{proof}

Complex structure $\JJ$ is  said to be \emph{Hermitian} if it preserves the metric:
\begin{align*}
 \langle \JJ u, \JJ v\rangle=\langle u, v\rangle, \enskip u, v\in\gg.
\end{align*}

In the fixed basis where $\JJ_0$ has the form~\eqref{eq:complex}, the direct computation shows that the following lemma holds.

\begin{lemma}\label{le:herm}
For the Hermitian complex structure $\JJ_0$ the corresponding Riemannian metric is given by~\eqref{eq:metrike} with $\omega_4 = 1$ and $\bar S_4$ being positive definite matrix of dimension $2n\times 2n$ satisfying \mbox{$(\bar S_4)_{i(n+i)} = 0$,} $i=1,\dots n$, and $J\bar S_4=\bar S_4 J$.
\end{lemma}

Now, we describe all $\JJ _0$-invariant closed 2-forms on the Lie algebra $\gg= \thn$.

In Section~\ref{sec:prelim} we introduced the standard basis of Lie algebra $\thn$:
\begin{align*}
e_1, \dots , e_n, f_1, \dots , f_n, z^*, e_1^*, \dots , e_n^*, f_1^*, \dots , f_n^*, z.
\end{align*}
Denote the dual basis on $\gg ^* = ({\thn})^*$ by
\begin{align}\label{eq:formebaza}
e^1, \dots , e^n, f^1, \dots , f^n, \zeta ^*, e^1_*, \dots , e^n_*, f^1_*, \dots , f^n_*, \zeta.
\end{align}
\begin{theorem}\label{th:forme}
	All $\JJ _0$ invariant closed 2-forms on $\thn$, $(n>1)$ are of the form
	\begin{align}\label{eq:forme}
		\Omega  =\, &  A^1_{ij}(e^i\wedge e^j + f^i\wedge f^j ) + A^2_{ij} e^i\wedge f^j + K_{ij}(e^i\wedge e^j_* + f^i\wedge f^j_*) \nonumber\\
		& +   d_i\, e^i\wedge f^i_* - \frac{\mu}{2} (e^i\wedge e^i_* + f^i\wedge f^i_*)  + \mu \, \zeta ^* \wedge \zeta,
	\end{align}
where $A^1_{ij} = - A^1_{ji}$,  $A^2_{ij} =  A^2_{ji}$,  $K_{ij} = -K_{ji}$, $d_i\in \RR$, $\mu \in \RR$ and summation is assumed over repeated indices.

The dimension of the space of such forms is $\frac{3n^2 + n + 2}{2}$, $n>1.$
\end{theorem}
\begin{proof}
	The commutators~\eqref{eq:thnkom2} can be written in the form
\begin{align*}
		de^i &=0, \enskip d f^i =0, & d\zeta ^* &= 0, & &\\
	de^i_* &= f^i\wedge \zeta ^*,   & df^i_* &= -e^i\wedge \zeta ^*, & d\zeta &= e^i\wedge f^i.
\end{align*}
The proof is by straightforward computation and will be omitted. Start from the general 2-form $\Omega$ over the basis~\eqref{eq:formebaza}. Impose the condition~\eqref{eq:Jinvariant} for complex structure $\JJ = \JJ _0$ given by~\eqref{eq:complex}. Finally, the condition $d\Omega =0$ bring us to the form~\eqref{eq:forme}.
\end{proof}
	
\begin{remark}\label{re:forme}
	For $n=1$, i.e. in case of 6-dimensional Lie algebra $T^*\mathfrak{h}_{3}$ the result is obtained in~\cite{SVB}.  The dimension of $\JJ$-invariant closed 2-forms is five. Such forms has general for~\eqref{eq:forme} with additional  terms
\begin{align*}
	a_1(e^1\wedge \zeta ^* - f^1\wedge \zeta) + a_2(f^1\wedge \zeta ^* + e^1\wedge \zeta),\enskip  a_1, a_2\in \RR.
\end{align*}
	These terms come from a different decomposition of the space of 3-forms in basis~\eqref{eq:formebaza}.
\end{remark}

As noted in the Introduction, the complex structure $\JJ_0$ is not abelian, hence the Hermitian metric from Lemma~\ref{le:herm} cannot be pseudo-K\"{a}hler (see~\cite[Theorem A]{Benson} and~\cite{CFG}). However, we show that there exists a family of Ricci-flat pseudo-K\"{a}hler metrics.

\begin{theorem}\label{th:pseudoK}
The Lie algebra $\thn$ admits Ricci-flat pseudo-K\"{a}hler metrics that are not flat. Every pseudo-K\"{a}hler metric on $\thn$ is equivalent to  $S=-\JJ _0\Omega$, where $\JJ _0$ is complex structure~\eqref{eq:complex} and $\Omega$ is symplectic form given by~\eqref{eq:forme}.
\end{theorem}
\begin{proof}
  The form of the pseudo-K\"{a}hler metric follows directly from~\eqref{eq:herm}. Since, we have already fixed the basis in a way that the complex form $\JJ$ takes the form~\eqref{eq:complex}, the choice of automorphisms preserving $\JJ_0$ is quite restricted. The obtained simplification is insignificant comparing to the difficult notation required, hence we choose not to perform it.

  Next, we need to examine the curvature properties of those metrics. The Ricci tensor $\rho$ on the nilpotent Lie group can be expressed in terms of operators $\ad$ and $\ad^*$:
  \begin{align*}
    \rho(u,v) = -\frac{1}{4}\tr(j_u\circ j_v)-\frac{1}{2}\tr(\ad_u\circ\ad^*_v),
  \end{align*}
  where $j_u v=\ad^*_v u$ for arbitrary left invariant vector fields $u, v$. By direct calculation we get that both summands are equal to zero, hence the metric is Ricci-flat.   Now, the only thing left is to show that at least one component of curvature tensor $R$ is non-zero. For example, long, but straightforward computation gives us that $R(e_1, f_1)$ contains non-trivial components.
\end{proof}

\begin{thank}
This research was supported by the Science Fund of the Republic of Serbia,
Grant No. 7744592, Integrability and Extremal Problems in Mechanics, Geometry
and Combinatorics - MEGIC.
\end{thank}

\end{document}